\documentclass[12pt]{amsart}
\usepackage{amsmath,amsthm,amssymb}
\usepackage[matrix,arrow,curve]{xy}
\makeatletter\@addtoreset{equation}{section} \makeatother

\newtheorem{theorem}[equation]{Theorem}
\newtheorem{proposition}[equation]{Proposition}
\newtheorem{lemma}[equation]{Lemma}
\newtheorem{corollary}[equation]{Corollary}

\theoremstyle{definition}
\newtheorem{example}[equation]{Example}
\newtheorem{definition}[equation]{Definition}

\theoremstyle{remark}
\newtheorem{remark}[equation]{Remark}

\newcommand{\ang}[1]{\langle #1 \rangle}
\newcommand{\ov}[1]{\overline{#1}}
\renewcommand{\tilde}[1]{\widetilde{#1}}

\DeclareMathAlphabet{\mathbbold}{U}{bbold}{m}{n}
\def \k {\mathbbold{k}}
\def \Xk {X_{\ov{\k}}}

\def \O {\mathcal{O}}
\def \H {\mathcal{M}}
\def \PP {\mathcal{P}}
\def\D {\mathcal{D}}
\def\X {\mathcal{X}}
\def \P {\mathbb{P}}
\def \Q {\mathbb{Q}}
\def \C {\mathbb{C}}
\def \Z {\mathbb{Z}}

\def \ge {\geqslant}
\def \le {\leqslant}
\def \kappa {\varkappa}

\def \numeq {\equiv}
\def \iso {\simeq}

\def \supp {\mathrm{supp}\,}
\def \mult {\mathrm{mult}}
\def \Bs {\mathrm{Bs}\,}
\def \NS {\mathrm{NS}_{\Q}}
\def \Bir {\mathrm{Bir}}
\def \Aut {\mathrm{Aut}}
\def \Pic {\mathrm{Pic}\,}
\def \Sing {\mathrm{Sing}\,}
\def \ord {\mathrm{ord}}
\def \id {\mathrm{id}}
\def \cchar {\mathrm{char}}

\def \t {~---\ }
\def\nlb {\nolinebreak}

\newcommand{\tit}{BIRATIONAL AUTOMORPHISMS OF NODAL QUARTIC 
THREEFOLDS}
\author{Constantin Shramov}
\title{\tit}
\address{University of Edinburgh, Kings Buildings, Mayfield Road,
Edinburgh EH9 3JZ, UK}
\thanks{The work was partially supported by RFFI grants
No.~05-01-00353 and No.~08-01-00395-a and grant N.Sh.-1987.2008.1}
\email{shramov@mccme.ru}

\begin{document}

\maketitle

\begin{abstract}
It is well-known that a nonsingular minimal cubic surface is 
birationally rigid; the 
group of its birational selfmaps is generated by biregular selfmaps
and birational involutions such that all relations between the latter are 
implied by standard relations between reflections on an elliptic curve. It is
also known that a factorial nodal quartic threefold is birationally rigid and
its group of birational selfmaps is generated by biregular ones and certain 
birational involutions. We prove that all relations between these involutions 
are implied by standard relations on elliptic curves, complete the proof of 
birational rigidity over a non-closed field and describe the situations when 
some of the birational involutions in question become regular (and, in 
particular, complete the proof of the initial theorem on birational rigidity,
since some details were not established in the original paper of M.\,Mella).
\end{abstract}

\section{Introduction}

One of the popular problems of birational geometry is
to find all Mori fibrations 
birational to a given Mori fibration
$\X\to\nlb T$, and to compute the group
of birational automorphisms $\Bir(\X)$ of a variety~$\X$.
The cases when there are few structures of Mori fibrations on $\X$
are of special interest,
for example, when there is only one such structure 
up to a natural equivalence: such
varieties are called \emph{birationally rigid} 
(see section~\ref{section:preliminaries} for a definition). 

The first example of a birationally rigid variety is a minimal cubic surface.
Recall that an Eckardt point on a cubic surface $S$ defined over a field $\k$
is a point contained in three lines lying on $S_{\ov{\k}}$.

\begin{theorem}[{see~\cite[Chapter~V, Theorems~1.5 and~1.6]{Manin}}]
\label{theorem:Manin}
Let $S$ be a nonsingular minimal cubic
surface over a perfect field $\k$. Then 

1. $S$ is birationally rigid;

2. $\Bir(S)$ is generated by its subgroup $\Aut(S)$, birational involutions 
$t_P$ centered in non-Eckardt
points (Geiser involutions) and birational involutions $t_{PQ}$ centered 
in pairs of conjugate points 
such that the corresponding line does not intersect any line contained 
in $S_{\ov{\k}}$ (Bertini involutions);

3. all relations between these generators are implied by the following:
\begin{gather*}
t_P^2=t_{PQ}^2=\id,\\
wt_Pw^{-1}=t_{w(P)} \mbox{\ for\ } w\in\Aut(S), \\
wt_{PQ}w^{-1}=t_{w(P)w(Q)} \mbox{\ for\ } w\in\Aut(S),\\
(t_{P_1}\circ t_{P_2}\circ t_{P_3})^2=\id 
\mbox{\ for collinear points\ } P_1, P_2, P_3.
\end{gather*}
\end{theorem}

Fano threefolds of low degree give examples of birationally rigid
varieties
with relatively simple groups of birational selfmaps.
Birational superrigidity (see section~\ref{section:preliminaries} for a 
definition) of a smooth quartic was proved
in~\cite{IskovskikhManin};
a proof of birational superrigidity of a smooth double cover
of $\P^3$ branched over a sextic and birational rigidity of a
smooth double cover of a quadric branched over a quartic section
(together with the calculation of its group of birational automorphisms)
can be found in~\cite{Iskovskikh-rigid} and in~\cite{IskovskikhPukhlikov}.

The same questions may be posed (and sometimes solved) for varieties with
mild singularities (for example, some nodal
varieties, see \cite{Pukhlikov-quartic}, 
\cite{CheltsovPark}, \cite{Grinenko} and~\cite{Mella}).

\begin{theorem}[{see~\cite[Theorem~2 or Theorem~7]{Mella}}]
\label{theorem:Mella-rigid}
Let $X$ be a factorial nodal\footnote{
See section~\ref{section:notation} for definitions.}
quartic threefold. Then 

1. $X$ is
birationally rigid, 

2. $\Bir(X)$ is generated by its subgroup $\Aut(X)$, 
birational involutions~$\tau_P$ centered
in singular points $P\in\Sing X$, and birational involutions~$\tau_L$
centered in lines\footnote{
A description of these will follow in section~\ref{section:generators}.}
$L$ containing one or two singular points of~$X$.
\end{theorem}

\begin{remark}
Note that conditions of Theorem~\ref{theorem:Mella-rigid} are indeed necessary.
If one allows more complicated singularities, the statement may fail to hold: 
for example, a general quartic hypersurface with a single singularity 
analytically isomorphic to a hypersurface singularity $xy+z^3+t^3=0$ is
factorial but not birationally rigid (see~\cite{CortiMella}). On the 
other hand, if one releases the factoriality assumption, $X$ may even be
rational, like a general determinantal quartic (see~\cite{Mella}).
In general factoriality is a global property that depends on the configuration
of singular points on~$X$, but there are sufficient conditions for $X$ to be
factorial depending only on the number of singular points 
(see~\cite[Theorems~1.2 and~1.3]{Cheltsov-quartic}, 
\cite[Theorem~1.3]{Shramov}). For a treatment of geometry 
of non-factorial nodal quartics see~\cite{Kaloghiros} 
(and also~\cite{Cheltsov-quartic} and~\cite{CheltsovGrinenko}).
\end{remark}

Recall that involutions $t_P\in\Bir(S)$ (resp., $t_{PQ}\in\Bir(S)$)
are also defined for ``bad'' points (resp., pairs of points), i.\,e. Eckardt
points $P$ (resp., pairs $\{P, Q\}$ such that the corresponding
line intersects some line contained in $S_{\ov{\k}}$)\t
but such involutions are regular on $S$.

Motivated by the analogy with a cubic surface, we give the following 
definitions for a (nodal factorial) quartic threefold $X$ defined over 
a field $\k$.

\begin{definition}[{cf., for example,~\cite[8.8.3]{Manin}
and~\cite[Definition~2.3]{CheltsovPark-Eckardt}}]
\label{definition:Eckardt-point}
Let $P$ be a singular point on $X$.
We call $P$ \emph{an Eckardt point}
if $P$ is a vertex of some (two-dimensional) cone contained in $\Xk$.
\end{definition}

\begin{definition}
Let $L\subset X$ be a line. We call $L$ \emph{an Eckardt line}
if there are infinitely many lines intersecting $L$ on $\Xk$.
\end{definition}

We prove the following result that describes regularizations on a quartic
threefold.

\begin{proposition}\label{proposition:regularisation}
Let $X$ be a factorial nodal quartic threefold. Then an
involution $\tau_P$ is regular on $X$ if and only if $P$ is 
an Eckardt point, and an involution $\tau_L$
is regular on $X$ if and only if $L$ is an Eckardt line.
\end{proposition}

\begin{remark}\label{remark:Mella-not-enough}
Actually, Theorem~\ref{theorem:Mella-rigid} is not exactly what
is proved in~\cite{Mella}. To derive Theorem~\ref{theorem:Mella-rigid}
from the results of~\cite{Mella} 
one needs to prove that Eckardt points and Eckardt lines
cannot be non-canonical centers on $X$ 
(see Remark~\ref{remark:why-Mella-not-enough}). Still, this is not hard to 
do; it is done in Remark~\ref{remark:why-Mella-not-enough}.
\end{remark}

As in Theorem~\ref{theorem:Manin}, 
one can observe that the involutions $\tau_P$ and $\tau_L$ may not be 
independent in $\Bir(X)$ because of relations arising from
standard ones for reflections on elliptic curves
(see Examples~\ref{example:3-points-relation} 
and~\ref{example:line-2-points-relation}).

The main goal of this paper 
is to prove the following result, which may be considered 
a generalization of the third part of Theorem~\ref{theorem:Manin}.

\begin{theorem}\label{theorem:relations}
In the setting of Theorem~\ref{theorem:Mella-rigid} 
all relations between the generators of $\Bir(X)$
are implied by the following ones:
\begin{gather*}
\tau_P^2=\tau_L^2=\id,\\
w\tau_Pw^{-1}=\tau_{w(P)} \mbox{\ for\ } w\in\Aut(S),\\ 
w\tau_Lw^{-1}=\tau_{w(L)} \mbox{\ for\ } w\in\Aut(S),\\
(\tau_{P_1}\tau_{P_2}\tau_{P_3})^2=\id 
\mbox{\ for collinear points\ } P_1, P_2, P_3,\\
(\tau_{P_1}\circ\tau_{P_2}\circ\tau_L)^2=\id \mbox{\ for\ } P_1, P_2\in L.
\end{gather*}
\end{theorem}

Note that one of possible generalizations of a quartic threefold is a 
Fano threefold
hypersurface of index $1$ with terminal singularities
in a weighted projective space. There are $95$ families of such
hypersurfaces. Their birational rigidity is known under some generality
assumptions (see~\cite[Theorem~1.3]{CPR}), as well as the fact that 
the groups of their birational automorphisms is generated by involutions
centered in points and lines (also known as Geiser and 
Bertini involutions or quadratic and elliptic involutions, 
see~\cite[Remark~1.4]{CPR}). The relations between these generators are 
also known and are analogous to those listed in Theorem~\ref{theorem:relations}
(see~\cite[Theorem~1.1]{CheltsovPark-weighted}).
Note that we establish the same results for a quartic without
any generality assumptions.

The paper is organized as follows. In section~\ref{section:notation}
we recall some standard definitions and fix notations that we are going to use 
throughout the paper. In section~\ref{section:preliminaries}
we recall standard definitions and constructions related to the method 
of maximal singularities.  
Section~\ref{section:auxiliary} contains some auxiliary results.
Section~\ref{section:generators} contains explicit 
description of the involutions $\tau_P$ and $\tau_L$ and obvious relations 
between them, and section~\ref{section:action} gathers information about the
action of these involutions.
Section~\ref{section:regularisation} contains a proof of
Proposition~\ref{proposition:regularisation} and a small improvement of 
the proof of Theorem~\ref{theorem:Mella-rigid} 
(see Remark~\ref{remark:why-Mella-not-enough}).
In section~\ref{section:max-centers} we prove 
Proposition~\ref{proposition:2-centers}, which is a technical counterpart 
of Theorem~\ref{theorem:relations}; actually, the method 
that reduces Theorem~\ref{theorem:relations} to 
Proposition~\ref{proposition:2-centers} is standard
(see~\cite[Chapter~V, \S 7.8]{Manin} 
or~\cite[3.2.4]{IskovskikhPukhlikov}), so we omit this step. 
Finally, section~\ref{section:non-closed-fields} contains an improvement of 
the proof of~\cite[Theorem~5]{Mella} (which states that 
Theorem~\ref{theorem:Mella-rigid} holds over algebraically non-closed 
fields as well).

I am grateful to I.\,Cheltsov for numerous explanations, to 
S.\,Galkin, Yu.\,G.\,Prokhorov and A.\,Kuznetsov for useful discussions and to 
I.\,Karzhemanov, J.\,Park, V.\,Przhijalkowsky, D.\,Ryder and D.\,Stepanov 
for remarks.
Part of this work was completed when I was staying in 
Max-Plank-Institute f\"ur Mathematik in October--November 2007. 
I am grateful to the staff of MPIM for their hospitality.

\section{Notation and conventions}
\label{section:notation}
All varieties throughout the paper are assumed to be defined over 
the field of complex numbers $\C$, except in 
section~\ref{section:non-closed-fields} where everything is defined over 
an arbitrary field $\k$ of characteristic $\cchar(\k)=0$.
On the other hand, 
all the results stated over $\C$ hold over $\k$ as well if the obvious 
changes are made to their statements.\footnote{
These are easy but not completely automatic. For example, 
in Remark~\ref{remark:no-line-3-points} 
the points $P_1$, $P_2$ and $P_3$ are 
not necessarily defined over $\k$ and one should assume only that they are 
contained in a line $L\subset\Xk$; in Lemma~\ref{lemma:2-points}
the line $L$ is not necessarily defined over $\k$ etc.}

Let $Y$ be a (projective, irreducible and normal) $n$-dimensional variety. 
A singular point $y\in Y$ is called \emph{an ordinary
double point} (or \emph{a node}) if its neighborhood is analytically 
isomorphic to a neighborhood of a vertex of a cone over a nonsingular
quadric of dimension $n-1$. 
If $Y$ is a hypersurface in $\P^{n+1}$ given by an equation $f=0$ in an affine 
neighborhood of $y$ then this 
property is equivalent to non-degeneracy of the Hessian matrix $H(f)$ at $y$.   
A variety that has only nodes as singularities is called \emph{nodal}.

A variety $Y$ is called \emph{factorial} if any Weil divisor on $Y$ is  
Cartier, and \emph{$\Q$-factorial} if an appropriate multiple 
of any Weil divisor is Cartier. 
Factorial varieties enjoy some properties typical for non-singular ones, 
for example, the Lefschetz theorem (see Lemma~\ref{lemma:Lefschetz}).
Note that for nodal varieties 
being factorial is equivalent to being $\Q$-\nlb factorial.
In the sequel by ``divisor'' we usually mean ``$\Q$-divisor''.

We use the following standard notation throughout the paper. 
If $D$ is a divisor and $\D$ is a linear system on $Y$, then $\supp D$ 
denotes the support of $D$, and $\Bs\D$ the base locus of $\D$. 
If $Z$ is a cycle,
$\mult_Z D$ denotes the multiplicity of $D$ at $Z$.
In fact we will use this notion only for the cases when $Z$ is either an
ordinary double point or a cycle not contained in the singular locus 
$\Sing Y$ of $Y$; under these assumptions $\mult_Z D$ may be defined using
the equation
$$\pi^*D=\pi^{-1}D+(\mult_Z D)E,$$
where $\pi:\tilde{Y}\to Y$ is the blow-up of $Z$ and $E$ is the (unique) 
exceptional divisor. The multiplicity $\mult_Z \D$ is defined as that of a 
general divisor $D\in\D$. 

The symbol ${}\numeq{}$ denotes numerical equivalence (of 
Cartier or $\Q$-\nlb Cartier divisors). If $S$ is a surface, we write 
$\NS^1(S)$ for the $\Q$-vector space generated by the Cartier divisors on 
$S$ modulo numerical equivalence; this space is endowed with a 
bilinear symmetric intersection form. 

If $C\subset\P^2$ is a (nonsingular) cubic curve, a 
\emph{group law on $C$} means
a standard group law on the elliptic curve with an inflection point of $C$
(any of these) as a zero element. Given such a curve $C$ and a point $P\in C$,
\emph{reflection with respect to $P$} means a reflection $R_P:C\to C$ 
with respect to the group law (i.\,e. a map $x\mapsto 2P-x$; recall 
that~$R_P$ depends only on the class of $P$ modulo $2$-torsion and does not 
depend on the choice of a zero element).
Since a projection from $P$ defines a double cover of $\P^1$, one can also 
associate to $P$ \emph{a Galois involution} $\tau_P$, i.\,e. the natural 
involution of this double cover; note that $\tau_P=R_{-\frac{P}{2}}$.

If $Y_1, \ldots, Y_k$ are subsets of $\P^n$, we denote by 
$\ang{Y_1, \ldots, Y_k}$ the linear span of $Y_1\cup\ldots\cup Y_k$.

We will reserve the symbol $X$ to denote a three-dimensional factorial
nodal quartic hypersurface throughout the paper. 

\section{Preliminaries on the method of maximal singularities}
\label{section:preliminaries}

We briefly recall the main constructions of the method of maximal
singularities and introduce the necessary notation and terminology
(see~\cite{Pukhlikov-essentials} or~\cite{Corti} for details).
The basic notions and facts concerning the Minimal Model Program and in 
particular necessary classes of singularities can be found 
in~\cite{Kollar} or~\cite{Matsuki}.

Let $V$ be a (three-dimensional) $\Q$-factorial Fano variety with terminal 
singularities and Picard number $\rho(V)=1$ (one could assume instead 
that $V$ is a Mori fibration over an arbitrary base $S$, 
but we do not need this level of generality). 
The variety $V$ is called \emph{birationally rigid} if 
for any birational 
map $\chi:V\dasharrow V'$ to a Mori fibration $V'\to S'$ 
the variety $V'$ is isomorphic to $V$ (and so $\chi$ is a birational
selfmap of $V$), 
and \emph{birationally superrigid} if it is birationally rigid and 
$\Bir(X)=\Aut(X)$ (see~\cite{Pukhlikov-essentials} or~\cite{Corti} 
for the definitions in the general case).

Let $V'\to S'$ be a Mori fibration. Assume that there is a
birational map $\chi:V\dasharrow V'$. There is an algorithm to obtain a
decomposition of $\chi$ into elementary maps (links) of four types, known
as the Sarkisov program (see, for example,~\cite{Corti}
or~\cite{Matsuki}). Choose a very ample divisor $M'$ on $V'$ and let
$\H=\chi_*^{-1}|M'|$ (note that $\H$ is mobile, i.\,e. 
has no base components, but in 
general has base points and is not complete). Let $\mu$ be
a rational number such that $\H\subset |-\mu K_V|$ (we will refer to $\mu$
as the \emph{degree} of the linear system $\H$). 
The N\"oether--Fano inequality
(see~\cite{Iskovskikh-rigid}, \cite{Corti}, \cite{Matsuki}
or~\cite{Pukhlikov-essentials})
implies that if $\chi$ is not an isomorphism then the pair $(V,
\frac{1}{\mu}\H)$ is not canonical. One can show that there is an extremal
contraction (in the sense of a usual Minimal Model Program)
$g:\tilde{V}\to V$, such that the discrepancy of the exceptional divisor
of $g$ with respect to the pair $(V, \frac{1}{\mu}\H)$ is negative. 
Furthermore, there exists a link $\chi_1$ of type
II or III (a definition can be found, for example, in~\cite{Corti} 
or~\cite{Matsuki})  starting with this contraction and decreasing an
appropriately 
defined ``degree'' of the map $\chi$ (i.\,e. the ``degree'' of 
$\chi\circ\chi_1^{-1}$ is less then that of $\chi$).
The only fact about this ``degree'' that we will 
use is the following: it decreases if the degree $\mu$ of the linear system  
$\H$ does (see~\cite{Corti} or~\cite{Matsuki} for details).

The previous statements imply the following: to prove that
$V$ cannot be transformed to another Mori fibration (i.\,e. is birationally
rigid) it suffices to check that there are no non-canonical 
centers\footnote{
To be more accurate, one should speak about non-canonical centers
\emph{with respect to $\frac{1}{\mu}\H$}. But we will  
avoid mentioning $\H$ 
since in all arguments that we use the linear system is fixed.}
on $V$ except those associated 
with links that give rise to birational automorphisms of $V$,
and to describe all birational selfmaps $\chi:V\dasharrow V$
it is sufficient to classify all non-canonical centers and to find 
an ``untwisting'' selfmap for each of them (i.\,e. a selfmap 
$\chi_Z$ such that the degree $\mu$ of $\H$ decreases after one applies 
$\chi_Z$, provided that $Z$ was a non-canonical center).

\section{Auxiliary statements}
\label{section:auxiliary}

We will refer to the following lemma as the Lefschetz theorem, since it is a 
straightforward analogue for factorial Fano varieties.

\begin{lemma}\label{lemma:Lefschetz}
Let $Y\subset\P^n$, $n\ge 4$, be a factorial hypersurface. Then any (effective) 
Weil divisor $D\subset Y$ is cut out by a hypersurface $\tilde{D}\subset\P^n$.
In particular, $\deg D$ is divisible by $\deg Y$.
\end{lemma}
\begin{proof}
A standard argument (see, for example,~\cite[Theorem~7.7]{Danilov}) shows
that a natural map $H^2(\P^n, \Z)\to H^2(Y, \Z)$ is an isomorphism.
On the other hand, since $H^1(Y, \O_Y)=H^2(Y, \O_Y)=0$ 
for any hypersurface in $\P^n$ with $n\ge 4$, one has $\Pic(Y)=H^2(Y, \Z)$.
Since $Y$ is factorial, any Weil divisor $D$ is Cartier, and the statement 
follows.
\end{proof}

The following results will be used in section~\ref{section:max-centers}.

\begin{theorem}[{see~\cite[Theorem~1.7.20]{Cheltsov-survey}}]
\label{theorem:sing-mult}
Let $V$ be a variety of dimension $\dim V\ge 3$, $x\in V$ an ordinary
double point and $D$ an effective divisor such that the pair $(V, D)$
is not canonical at $x$. Then $\mult_x D>1$.
\end{theorem}

\begin{lemma}[{cf.~\cite[Lemma~0.2.8]{CheltsovPark-Halphen}}]
\label{lemma:semi-definite}
Let $S$ be a nonsingular surface and $\Delta$ an effective divisor on
$S$ such that
$$\Delta\numeq\sum\limits_{i=1}^r c_i C_i,$$
where $c_i>0$ and
the support of $\Delta$ does not contain any of the curves~$C_i$.
Assume that the intersection form on the subspace
$W\nlb\subset\nlb\NS^1(S)$ generated by the curves $C_i$ is 
negative semidefinite. Then $\Delta^2=0$.
\end{lemma}
\begin{proof}
The argument is identical to that of Lemma~0.2.8
in~\cite{CheltsovPark-Halphen}. Let $\Delta=\sum_{j=1}^k b_jB_j$, $b_i>0$.
Then
$$
0\ge (\sum\limits_{i=1}^r c_iC_i)^2
=(\sum\limits_{j=1}^k b_jB_j)(\sum\limits_{i=1}^r c_iC_i)\ge 0,
$$
that is,
$$0=(\sum\limits_{i=1}^r c_iC_i)^2=\Delta^2.$$
\end{proof}

\begin{lemma}\label{lemma:tangent-plane}
Let $L\subset Y\subset\P^4$ be a line inside a nodal quartic.
Then the following conditions are equivalent:

(i) there is a hyperplane $H$ tangent to $Y$ along $L$,

(ii) there are infinitely many planes $\Pi$ such that 
$\left.Y\right|_{\Pi}=2L+Q$ for some (possibly reducible) conic $Q$,

(iii) $L$ contains three singular points of $Y$.

Moreover, if one of these conditions holds then 
$\mult_L H=2$ and any plane $\Pi$ as in 
$(ii)$ is contained in $H$. 
\end{lemma}
\begin{proof}
Easy.
\end{proof}

The following lemmas describe the singularities of general
hyperplane sections of a threefold nodal hypersurface.

\begin{lemma}\label{lemma:general-section}
Let $Y\subset\P^4$ be a nodal hypersurface, $P$ a singular point of $Y$ and
$\Pi_0\ni P$ a two-dimensional plane. Assume that
$$\left.Y\right|_{\Pi_0}=\sum m_i C_i+\sum m_j' C_j',$$
where $P\not\in C_j'$, $P\in C_i$,
the curves $C_i$ are nonsingular at $P$, and $m_i$, $m_j'$ are 
integers. Let
$k=\sum m_i$. Take a general hyperplane section $H\subset Y$ passing
through $\Pi_0$.
Then the singularity $P\in H$ is Du Val of type $A_{k'}$ with
$k'\le k-1$.
\end{lemma}
\begin{proof}
Choose an affine neighborhood $U$ of $P$ with coordinates $x, y, z, 
t$ so that the hypersurface $Y$ is given by an equation $F(x, y, z, t)=0$,
where
$$F(x, y, z, t)=xz+yt+F_{{}\ge 3}(x, y, z, t),$$ 
and $\ord_0 (F_{{}\ge 3})\ge 3$.
If the restriction of the polynomial $xz+yt$ to $\Pi_0$ is not identically
zero, then $H$ has an ordinary double point (that is, a Du Val singularity
of type $A_1$) at $P$. Hence we may assume that $\Pi_0$ is given by the
equation $z=t=0$, and $H$ is cut out by a hyperplane $t=\alpha z$. Then
$H$ is given by the equation
$$(x+\alpha y)z+\tilde{F}_{{}\ge 3}(x+\alpha y, y, z)=0,$$
where $\ord_0 (\tilde{F}_{{}\ge 3})\ge 3$, 
and hence $H$ has a Du Val singularity of
type $A_{k'}$ at $P$ (see, for example, \cite[Chapter II, 11.1]{AVGZ}).

Assume that $k'\ge 2$. The projectivization of the plane $\Pi_0$ gives
a line $l$ contained in a nonsingular quadric
$Q=(xz+yt=0)\subset\P(V)\iso\P^3$, and the projectivization of the
hyperplane $t=\alpha z$ gives a plane in $\P(V)$,
intersecting $Q$ by a pair of lines $l\cup l'$.
Let $\ov{f}:\ov{Y}\to Y$ be a blow-up of the point $P$ with an exceptional
divisor $E$, and $\ov{H}=\ov{f}^{-1}H$.
Let $\ov{f}_H$ be a restriction of $\ov{f}$ to $\ov{H}$.
Then $\ov{f}$ is a blow-up of $H$ at the point $P$, and the exceptional
locus of $\ov{f}_H$ is identified with $E\cap\ov{H}=l\cup l'$.
The surface $\ov{H}$ has a Du Val singularity of type $A_{k'-2}$ at the
point $P'=l\cap l'$ and is nonsingular at the points $(l\cup
l')\setminus\{P'\}$. The proper transforms $\ov{f}_H^{-1}C_i$ of the
curves $C_i$ intersect the line $l$ and do not pass through $P'$.

Consider a resolution of singularities
$f:H'\to H$ that is obtained from $\ov{H}$ by a sequence of blow-ups. Let
$l_1, \ldots, l_{k'}$ be exceptional curves of the resolution $f$ that are
contracted to $P$, labelled so that
$l_il_{i+1}=1$ for $1\le i\le k'-1$.
According to the above observation, all proper transforms $f^{-1} C_i$
intersect one and the same exceptional curve, which corresponds to one
of the ends of the chain of exceptional curves (say,~$l_{k'}$) and is a 
strict transform of $l\subset\ov{H}$ on $H'$. 

Let us compute the multiplicities of the exceptional curves 
$l_t$ in the pull-back of the curve $C_i$. Let
$$f^* C_i=f^{-1}C_i+\sum\limits_{t=1}^{k'}a_{i, t}l_t.$$
From the system of equations
$$
0=l_tf^*C_i=
\begin{cases}
a_{i, 2}-2a_{i, 1} \text{ for } t=1,\\
a_{i, t+1}-2a_{i, t}+a_{i, t-1} \text{ for } 1<t<k',\\
1-2a_{i, k'}+a_{i, k'-1} \text{ for } t=k';
\end{cases}
$$
we obtain
$$a_{i, t}=\frac{t}{k'+1}.$$
In particular, for all $C_i$ we have
$$a_{i, 1}=\frac{1}{k'+1}.$$
Since $D=\sum m_iC_i+\sum m_j'C_j'$ is a Cartier divisor and hence the
divisor $f^*D$ is integral, one has $\frac{k}{k'+1}\in\Z$ and hence
$k\ge k'+1$.
\end{proof}

\begin{lemma}\label{lemma:general-section-mult-line}
Let $Y\subset\P^4$ be a nodal hypersurface of degree $\deg Y=d$ and
$L\subset Y$ a line, containing exactly $n$ singular points of
$Y$. Let $\Pi_0$ be a two-dimensional plane such that
$\left.Y\right|_{\Pi_0}=kL+C$, where $C\ge 0$ and $L\not\subset\supp C$.
Assume that $k\ge 2$, take a general hyperplane section $H\subset Y$
passing through $\Pi_0$ and let $$\PP=(L\cap\Sing H)\setminus (L\cap\Sing
Y).$$
Then
\begin{enumerate}
\item $H$ has isolated singularities, and for any point
$P_0\in L\setminus\Sing Y$ one can chose $H$ so that $H$ is
nonsingular at $P_0$;
\item $\PP$ contains at most $d-n-1$ points;
\item any point $P\in\PP$ is a Du Val singularity of type $A_{k-1}$ on
$H$. \end{enumerate}
\end{lemma}
\begin{proof}
The first assertion is obvious: it suffices to choose $H$ so that
the three-dimensional subspace $\ang{H}\iso\P^3$ does not coincide with a 
tangent
subspace $T_{P_0}Y\iso\P^3$ at $P_0\in L\setminus\Sing Y$.

Now choose homogeneous coordinates $x_0, \ldots, x_4$ in $\P^4$ such that
the subspace $\ang{H}$ is given by equation $x_4=0$, the plane
$\Pi_0$ by equations $x_3=x_4=0$, and the line $L$ by equations
$x_2=x_3=x_4=\nlb 0$. Then $Y$ is given by an equation of the form
$$x_2^kF(x_0, x_1, x_2)+x_3G_3(x_0, \ldots, x_3)+x_4G_4(x_0, \ldots,
x_4)=0,$$
where $\deg F=d-k$, $\deg G_3=\deg G_4=d-1$. The equation
of the surface $H$ in $\ang{H}\iso\P^3$ with homogeneous coordinates $x_0,
\ldots,x_3$ is
\begin{equation}\label{equation:DuVal1}
x_2^kF(x_0, x_1, x_2)+x_3G_3(x_0, \ldots, x_3)=0.
\end{equation}
Note that partial derivatives of the left hand side
of~\ref{equation:DuVal1} with respect to $x_0$, $x_1$ and $x_2$
vanish on the line $L$, 
hence the set $L\cap\Sing H$ is just a zero locus of the
restriction of the polynomial $G_3$ to $L$. Moreover, $G_3$ does not
vanish identically on $L$ since otherwise $H$ would be singular along $L$.
This implies the second assertion of the Lemma.

To prove the third assertion consider a point $P\in\PP$. We may assume
that $P=(1:0:0:0:0)$. By the first assertion of the Lemma
for any point $P'\in L\setminus\Sing Y$ there is a hyperplane section
nonsingular at $P'$; since $H$ is general, we may assume that
the surface $H$ is nonsingular at all the points
$P'\in (L\cap C)\setminus\Sing Y$, i.\,e. $P$ is not contained in
$L\cap C$ and hence $F$ is not of the form $F=x_1F_1+x_2F_2$.
Since $G_3$ does not vanish identically on $L$, it is not of the form
$G_3=x_2G_{32}+x_3G_{33}$. Choose an affine neighborhood $U$ of $P$; let
$x$, $y$, $z$ be coordinates in $U$ corresponding to (homogeneous)
coordinates $x_1$, $x_2$, $x_3$.
The surface $H$ in the neighborhood of $P$ is given by
\begin{equation}\label{equation:DuVal2}
y^k(1+\tilde{F}(x, y))+z(c_xx+c_yy+c_zz+\tilde{G}_3(x, y, z)),
\end{equation}
where $\ord_0\tilde{F}\ge 1$, $\ord_0\tilde{G}_3\ge 2$, $c_x$, $c_y$ and
$c_z$ are constants such that $c_x\neq\nlb 0$. It is easy to see that
the equation~\ref{equation:DuVal2} defines a Du Val singularity of type
$A_{k-1}$.
\end{proof}

\section{Generators and relations}\label{section:generators}

From now on we denote by $X$ a nodal factorial quartic threefold.
In this section we recall constructions of birational involutions that 
(together with $\Aut(X)$) generate $\Bir(X)$, and list obvious relations 
between them. Note that the generators of $\Bir(X)$ are constructed 
in a very standard way
(see, for example,~\cite[Introduction and Chapter~V,~1.4]{Manin}, 
\cite[3.1.2 and 3.1.4]{IskovskikhPukhlikov},
\cite[5.1.2 and 5.1.3]{IskovskikhPukhlikov},
\cite[2.6]{CPR}, \cite[Example~4.4]{Shramov}~etc).

\begin{example}
\label{example:point-involution}
Let $P$ be a singular point of $X$. Projection from $P$ defines a
(rational) double cover $\phi:X\dasharrow\P^3$; the Galois involution of 
$\phi$ gives rise to a birational involution $\tau_P$ of $X$.
\end{example}

\begin{example}
\label{example:line-1-point-involution}
Let $P$ be a singular point of $X$, and $L\subset X$ a line containing
$P$ and no other singular points of $X$. Projection from $L$ defines
an elliptic fibration $\psi:X\dasharrow\P^2$, and fiberwise 
reflection\footnote{To be more precise one should define the reflection on a
general fiber of (a~regularization of) $\psi$ and then extend it to 
an involution of the whole variety.}
in a section of $\phi$ arising from the point $P$ gives rise to a birational 
involution $\tau_L$ of $X$.
\end{example}

\begin{example}
\label{example:line-2-points-involution}
Let $P_1$ and $P_2$ be singular points of $X$, and $L\subset X$ a line 
passing through $P_1$ and $P_2$ but no other singular points of $X$.
As in Example~\ref{example:line-1-point-involution}, define an elliptic 
fibration $\psi$, denote by $E_1$ and $E_2$ the sections of its regularization 
corresponding to the points $P_2$ and $P_2$, and take a reflection 
(with respect to the group law on a general fiber)
in the section\footnote{\label{footnote:popolam}
Actually, since an elliptic curve contains $2$-torsion points, 
$\frac{E_1+E_2}{2}$ is not correctly defined as a section of the elliptic 
fibration, but the corresponding fiberwise reflection is correctly defined 
since it does not depend on $2$-torsion, so from here on we will allow such 
abuse of notation.}
$\frac{E_1+E_2}{2}$; one can also define this involution as 
a fiberwise Galois involution with respect to the section $-(E_1+E_2)$, i.\,e.
the section arising from $L$. We will also denote the corresponding 
birational involution by~$\tau_L$.
\end{example}

\begin{remark}
Note that the involution $\varphi^L_2$ defined in~\cite{Mella} in the setting
of Example~\ref{example:line-2-points-involution} is different from the
involution $\tau_L$ defined 
in Example~\ref{example:line-2-points-involution} (in~\cite{Mella}
it corresponds to a reflection in $E_1$). This does not matter if one 
is interested only in the structure of the group $\Bir(X)$ since
$\tau_L=\tau_{P_1}\circ\tau_{P_2}\circ\varphi^L_2$,
but our definition is slightly more natural from the point of view of 
Sarkisov program, since it is exactly the untwisting involution for $L$ in 
this case (see Lemma~\ref{lemma:line-2-points-action}). 
\end{remark}

\begin{remark}\label{remark:no-line-3-points}
A quartic with isolated singularities cannot have more than three collinear 
singular points. The situation of three singular points $P_1$, $P_2$ and 
$P_3$ contained in some line $L\subset X$ is possible, but such lines 
do not contribute to $\Bir(X)$ since they cannot be non-canonical centers 
(see~\cite{Mella} or use Lemma~\ref{lemma:tangent-plane}). 
Moreover, if one defines an involution $\tau_L$ 
in this situation as in Example~\ref{example:line-2-points-involution}
with respect to the points $P_1$ and $P_2$, it will coincide with the 
involution $\tau_{P_3}$.
\end{remark}

\begin{remark}\label{remark:all-on-elliptic-fibration}
Note that an involution $\tau_P$ also acts as a fiberwise reflection on
any elliptic fibration associated with a line $L\subset X$ containing~$P$
(one should reflect in the section $-\frac{E_P}{2}$, where $E_P$ is a 
section corresponding to $P$). 
\end{remark}

One of the main results of~\cite{Mella} states 
(see Theorem~\ref{theorem:Mella-rigid}) that the involutions 
listed in Examples~\ref{example:point-involution}, 
\ref{example:line-1-point-involution} 
and~\ref{example:line-2-points-involution} together with $\Aut(X)$ generate the 
group $\Bir(X)$. On the other hand, it is easy to see that there may be
relations between these generators.

\begin{example}\label{example:3-points-relation}
Let $P_1, P_2, P_3\in\Sing X$ be collinear. Then the line
$L=\ang{P_1, P_2, P_3}$ is contained in $X$, and all the involutions 
$\tau_{P_i}$ act fiberwise on the corresponding elliptic fibration.
Hence one has
\begin{equation}\label{equation:3-points-relation}
(\tau_{P_1}\circ\tau_{P_2}\circ\tau_{P_3})^2=\id
\end{equation}
by the well-known relation between three reflections on an elliptic curve
(see, for example,~\cite[Chapter I, 2.3]{Manin}).
\end{example}

\begin{example}\label{example:line-2-points-relation}
Let $P_1, P_2\in\Sing X$; let $L\subset X$ be a line containing $P_1$ and $P_2$ but no other singular points of $X$. Then all three involutions 
$\tau_L$, $\tau_{P_1}$ and $\tau_{P_2}$ act fiberwise on the elliptic fibration associated to $L$, so that
\begin{equation}\label{equation:line-2-points-relation}
(\tau_{P_1}\circ\tau_{P_2}\circ\tau_L)^2=\id.
\end{equation}
\end{example}

\begin{remark}\label{remark:index-permutation}
Note that there are other relations that differ 
from~\ref{equation:3-points-relation}
and~\ref{equation:line-2-points-relation} by a permutation of indices,
but they are equivalent to~\ref{equation:3-points-relation}
and~\ref{equation:line-2-points-relation} (modulo trivial relations
$\tau_{P_i}^2=\tau_L^2=\id$).
\end{remark}

One of the main goals of this paper is to show that 
the relations~\ref{equation:3-points-relation} 
and~\ref{equation:line-2-points-relation}  
imply all relations in $\Bir(X)$ (up to
trivial ones, see Theorem~\ref{theorem:relations}). This will be proved in
section~\ref{section:max-centers}.

%%%%%%%%%%%%%%%%%%%%%%%%%%%%%%%%%%%%%%%%%%%%%%%%%%%%%%%%%%%%%%%%%%%%%%%

\section{Action of birational involutions}
\label{section:action}

In this section we gather information about the action of the birational 
involutions $\tau_P$ and $\tau_L$, i.\,e. describe the way the degrees 
and multiplicities change under the action of these involutions.

We fix the following notations. Let $\chi:X\dasharrow X$ be a birational map,
and $\H=\H(\chi)$ be a linear system of degree $\mu(\chi)$ defined 
as in section~\ref{section:preliminaries}. For a subvariety 
$Z\subset X$ we put $\nu_Z(\chi)=\mult_Z\H(\chi)$.

\begin{remark}\label{remark:codim-2}
Assume that a line $L\subset X$ is not an Eckardt line, 
contains a singular point $P$ and at most one more singular 
point of $X$. Then there is only a finite number of conics and lines in the 
fibers of a projection $\psi$ from $L$: if a fiber is reducible, then it either 
contains lines intersecting $L$ and different from $L$ (by assumption 
there is only a finite number of fibers of this type), or 
it contains $L$, i.\,e. the corresponding plane section has multiplicity at 
least $2$ along $L$, which is possible for an infinite number of plane sections
only if $L$ contains three singular points of $X$ by 
Lemma~\ref{lemma:tangent-plane}. Moreover, only 
a finite number of irreducible residual 
cubic curves in plane sections passing through $L$
has a singular point at $P$, and in the case of two singular points of $X$ 
lying on $L$ none of these irreducible cubic curves has a singular point 
on $L$ outside the singular points on $X$. Hence 
the birational involutions $\tilde{\tau_P}$ and~$\tilde{\tau_L}$ 
(corresponding to $\tau_P, \tau_L\in\Bir(X)$) 
of the variety $\tilde{X}$, obtained 
as a blow-up of $X$ first at singular points lying on $L$ and then along
the strict transform of $L$, are regular in codimension $1$ 
since both a reflection and a Galois involution are well 
defined in a smooth point of an irreducible plane cubic.
\end{remark}

\begin{lemma}\label{lemma:line-1-point-action}
Let $L\subset X$ be a line containing a unique singular point $P$ of~$X$. 
Assume that $L$ is not an Eckardt line.
Then
\begin{gather*}
\mu(\chi\circ\tau_L)=11\mu(\chi)-10\nu_L(\chi),\\
\nu_L(\chi\circ\tau_L)=12\mu(\chi)-11\nu_L(\chi),\\
\nu_P(\chi\circ\tau_L)=6\mu(\chi)-6\nu_L(\chi)+\nu_P(\chi).
\end{gather*}
\end{lemma}
\begin{proof}
The proof is reduced to the calculation of the action of 
a birational involution $\tilde{\tau_L}$ corresponding to $\tau_L$ 
on the Picard 
group of the variety $\tilde{X}$ obtained as the blow-up first of $P$ 
and then of the strict
transform of $L$. Note that $\tilde{\tau_L}$ is regular on $X$ in 
codimension $1$ by Remark~\ref{remark:codim-2}. The rest of the 
calculation coincides with that of~\cite[Lemma~5.1.3]{IskovskikhPukhlikov}.
\end{proof}

\begin{lemma}\label{lemma:line-2-points-action}
Let $L\subset X$ be a line containing exactly two singular points of $X$,
say, $P_1$ and $P_2$.
Assume that $L$ is not an Eckardt line.
Then
\begin{gather*}
\mu(\chi\circ\tau_L)=5\mu(\chi)-4\nu_L(\chi),\\
\nu_L(\chi\circ\tau_L)=6\mu(\chi)-5\nu_L(\chi),\\
\nu_{P_1}(\chi\circ\tau_L)=3\mu(\chi)-3\nu_L(\chi)+\nu_{P_2}(\chi),\\
\nu_{P_2}(\chi\circ\tau_L)=3\mu(\chi)-3\nu_L(\chi)+\nu_{P_1}(\chi).
\end{gather*}
\end{lemma}
\begin{proof}
Analogous to that of Lemma~\ref{lemma:line-1-point-action}.
\end{proof}

\begin{lemma}\label{lemma:point-action-1-point}
Let $L\subset X$ be a line containing a unique singular point $P$ of~$X$.
Assume that $L$ is not an Eckardt line.
Then
\begin{gather*}
\mu(\chi\circ\tau_P)=3\mu(\chi)-2\nu_P(\chi),\\
\nu_P(\chi\circ\tau_P)=4\mu(\chi)-3\nu_P(\chi),\\
\nu_L(\chi\circ\tau_P)=\mu(\chi)-\nu_P(\chi)+\nu_L(\chi).
\end{gather*}
\end{lemma}
\begin{proof}
Note that $\tau_P$ preserves the elliptic fibration associated with $L$.
The rest is analogous to Lemma~\ref{lemma:line-1-point-action}.
\end{proof}

\begin{lemma}\label{lemma:point-action-2-points}
Let $L\subset X$ be a line containing exactly two singular points of $X$,
say, $P$ and $P_1$.
Assume that $L$ is not an Eckardt line.
Then
\begin{gather*}
\mu(\chi\circ\tau_P)=3\mu(\chi)-2\nu_P(\chi),\\
\nu_{P}(\chi\circ\tau_P)=4\mu(\chi)-3\nu_P(\chi),\\
\nu_{P_1}(\chi\circ\tau_P)=\mu(\chi)-\nu_P(\chi)+\nu_L(\chi),\\
\nu_L(\chi\circ\tau_P)=\mu(\chi)-\nu_P(\chi)+\nu_{P_1}(\chi).
\end{gather*}
\end{lemma}
\begin{proof}
Analogous to that of Lemma~\ref{lemma:point-action-1-point}.
\end{proof}

\begin{lemma}\label{lemma:point-action-3-points}
Let $L\subset X$ be a line containing three singular points of $X$,
say, $P$, $P_1$ and $P_2$.
Assume that $P$, $P_1$ and $P_2$ are not Eckardt points.
Then
\begin{gather*}
\mu(\chi\circ\tau_P)=3\mu(\chi)-2\nu_P(\chi),\\
\nu_{P}(\chi\circ\tau_P)=4\mu(\chi)-3\nu_P(\chi),\\
\nu_{P_1}(\chi\circ\tau_P)=\mu(\chi)-\nu_P(\chi)+\nu_{P_2}(\chi),\\
\nu_{P_2}(\chi\circ\tau_P)=\mu(\chi)-\nu_P(\chi)+\nu_{P_1}(\chi),\\
\nu_L(\chi\circ\tau_P)=2\mu(\chi)-2\nu_P(\chi)+\nu_L(\chi).
\end{gather*}
\end{lemma}
\begin{proof}
Analogous to that of Lemma~\ref{lemma:point-action-1-point}. We give a 
sketch to highlight some minor differences.

Let $\tilde{X}$ be the blow-up of $X$ in $P$, $P_1$, $P_2$ 
and then the strict transform of $L$, and $\tilde{\tau_P}$ the corresponding
(birational) involution of $\tilde{X}$ (note that $\tilde{\tau_P}$
is regular in codimension $1$ by~\ref{remark:codim-2}).
Let $h$ denote the class of a pull-back of a hyperplane section of $X$ in
$\Pic(\tilde{X})$, and let $e$, $e_1$, $e_2$ and $e_L$ denote the classes of
(the preimages of) exceptional divisors.
Note that $\tilde{X}$ has a  
structure of an elliptic fibration $\psi:\tilde{X}\to\P^2$. Let $C$ be a 
general fiber of $\psi$, and $S$ the preimage of a general line in $\P^2$.
Then the kernel $K$ of the restriction map $\Pic(\tilde{X})\to\Pic(C)$ 
is generated 
by $h-e-e_1-e_2-e_L$ and $e_L$: indeed, $K$ is generated by the preimage 
of a general line in $\P^2$ (that is $h-e-e_1-e_2-e_L$) and divisors 
swept out by the components of reducible fibers; one of the latter is $e_L$, 
and another is swept out by conics and is equivalent to $h-2e-e_1-e_2-e_L$ 
since a general conic is contained in a hyperplane section $H\subset X$ 
tangent to $X$ along $L$ and 
$\mult_L H=2$ by Lemma~\ref{lemma:tangent-plane}.
The remaining computations are analogous to those 
of~\cite[Lemma~5.1.3]{IskovskikhPukhlikov}. Restricting to $C$, one gets
\begin{gather*}
\tilde{\tau_P}^*h=3(e_1+e_2)-h+m_1(h-e-e_1-e_2-e_L)+m_2e_L,\\
\tilde{\tau_P}^*e=e_1+e_2-e+n_1(h-e-e_1-e_2-e_L)+n_2e_L,\\
\tilde{\tau_P}^*e_L=e_L+k_1(h-e-e_1-e_2-e_L)+k_2e_L,\\
\tilde{\tau_P}^*e_1=e_2+l_1(h-e-e_1-e_2-e_L)+l_2e_L,\\
\tilde{\tau_P}^*e_2=e_1+l_1(h-e-e_1-e_2-e_L)+l_2e_L.
\end{gather*}
Computing intersection numbers on $S$, one obtains that $l_1=l_2=0$,
$n_2=0$, $n_1=2$, $k_1=k_2=0$, $m_1=4$, $m_2=2$, and the statement follows.
\end{proof}

%%%%%%%%%%%%%%%%%%%%%%%%%%%%%%%%%%%%%%%%%%%%%%%%%%%%%%%%%%%%%%%

\section{Regularization}
\label{section:regularisation}

In this section we describe the cases when the birational 
involutions of $X$ become regular. These effects are analogous to 
regularization of birational involutions of minimal cubic surfaces 
arising from Eckardt points.

The following example shows that birational involutions of both types may 
regularize on $X$.

\begin{example}[{cf.~\cite[7.4.2]{CPR}}]\label{example:Eckardt-point}
Let $X\subset\P^4$ be given by equation
\begin{equation}\label{equation:Eckardt-point}
w^2q_2(x, y, z, t)+q_4(x, y, z, t)=0,
\end{equation}
where $(x:y:z:t:w)$ are homogeneous coordinates in $\P^4$ and $q_i$ is
a form of degree $i$. Let $P=(0:0:0:0:1)$; note that $P$ is a singular point on
$X$, and $X$ contains the cone $q_2=q_4=0$ with its vertex at $P$. 
 
Let $L\subset X$ be a line passing through $P$ such that $L$ contains 
no singular points of $X$ except $P$. 
It is easy to see that the involution $\tau_P$ is regular and acts as
$$\iota:(x:y:z:t:w)\mapsto (x:y:z:t:-w).$$ 
Moreover, let $\Pi$ be a general plane 
containing $L$, so that $\left. X\right|_{\Pi}=L\cup C$; then 
$C$ is a nonsingular plane cubic, and $P\in C$ is an inflection point, 
so the involutions $\tau_L$ and $\tau_P$
coincide on $C$ 
(and hence on $X$), and so $\tau_L$ is also regular on $X$.

If $q_4$ is sufficiently general, $P$ is a node and, moreover, the only singular
point on $X$. The latter implies that $X$ is factorial 
by~\cite[Theorem~1.2]{Cheltsov-quartic} 
(in particular, $X$ is birationally superrigid
by Theorem~\ref{theorem:Mella-rigid} and the previous argument).

If $X$ is given by equation
$$w^2(xy+zt)-(x^3y+y^4+z^4+t^4),$$
then $X$ is singular exactly in three collinear (ordinary double) points:
$P'=(1:0:0:0:1)$, $P''=(-1:0:0:0:1)$ and $P$. In particular,
$X$ is factorial by~\cite[Theorem~1.2]{Cheltsov-quartic} 
(and hence birationally rigid by Theorem~\ref{theorem:Mella-rigid}).
\end{example}

\begin{example}\label{example:Eckardt-line-regularisation}
Let $L\subset X$ be a line such that there are infinitely many lines 
contained in $X$ that intersect $L$ in smooth points of $X$. Then 
the involution $\tau_L$ is regular (provided that it is defined, i.\,e. 
$L$ contains one or two singular points of $X$). Indeed, assume that 
$\tau_L$ is not regular on $X$. Then there is
a mobile linear system $\H\subset|-\mu K_X|$
such that $L$ is a non-canonical center with respect to
$\frac{1}{\mu}\H$ 
(one can take $\H=(\tau_L)_*^{-1}|\O(1)|$), i.\,e. $\mult_L\H>\mu$.
In particular, $\mult_P\H>\nlb\mu$, and hence all lines passing through $P$
are contained in $\Bs\H$, a~contradiction. 
\end{example}

The next example shows that there are factorial nodal quartics containing lines
of the type described in Example~\ref{example:Eckardt-line-regularisation}.

\begin{example}\label{example:smooth-Eckardt-point}
Let $X\subset\P^4$ be given by the equation
$$
w^3x+wx(xy+zt)+(x^4+y^4+z^4+tz^3)=0.
$$
Then $P=(0:0:0:0:1)$ is the vertex of a two-dimensional cone contained in 
$X$, and the only singular point of $X$ is a node at $P'=(0:0:0:1:0)$;
in particular, $X$ is factorial by~\cite[Theorem~1.2]{Cheltsov-quartic}.
The line $L=\ang{P, P'}$ is contained in $X$ and fits 
into the setting of Example~\ref{example:Eckardt-line-regularisation}.
\end{example}

\begin{remark}\label{remark:Eckardt-point-easy}
If $P\in\Sing X$ is a point such that there are infinitely many lines 
contained in $X$ and passing through $P$, one could also argue as in 
Example~\ref{example:Eckardt-line-regularisation} using 
Theorem~\ref{theorem:sing-mult} to show that $P$ cannot
be a non-canonical center and hence $\tau_P$
is regular. 
\end{remark}

We will see below that Examples~\ref{example:Eckardt-point} 
and~\ref{example:Eckardt-line-regularisation}
describe (at least to some extent) the general situation.

\begin{lemma}\label{lemma:Eckardt-point-equation}
Let $X$ have a singular Eckardt point $P$. 
Then $X$ is given by an equation
of type~\ref{equation:Eckardt-point}; moreover, any line $L\subset X$ 
passing through
$P$ contains either one or three singular points of $X$.
\end{lemma}
\begin{proof}
Let $(x:y:z:t:w)$ be homogeneous coordinates in $\P^4$ such that 
$P=(0:0:0:0:1)$. Then $X$ is given by  
\begin{equation}\label{equation:pre-Eckardt-point}
w^2q_2(x, y, z, t)+wq_3(x, y, z, t)+q_4(x, y, z, t)=0,
\end{equation}
where $q_i$ is a form of degree $i$. 

Assume that $q_3$ is not divisible by $q_2$. The equation $q_2=0$ defines a 
nonsingular quadric surface in $\P=(w=0)\iso\P^3$. By assumption the curves 
cut out on this quadric by $q_3=0$ and $q_4=\nlb 0$ have a common 
(irreducible) component $F$ (so that $K$ is a cone over~$F$). By the Lefschetz 
theorem $\deg K$ must be divisible 
by $4$; since $\deg K=\deg F\le\nlb 6$, the only possible case is $\deg F=4$,
i.\,e. $F$ is an irreducible curve of type $(2, 2)$. In the latter case $K$
is cut out on $X$ by a hyperplane (again by the Lefschetz theorem), and hence
$F\subset\P$ is contained in a plane, a contradiction. 

So $q_3=q_2\cdot l$ for some linear form $l$, and replacing 
$w$ by $w+\frac{l}{2}$ we may assume that $q_3=0$ 
and $K$ is given by equations $q_2=q_4=0$.

Now assume that a line $L\subset X$ passing through $P$ 
contains a point $P'\in\Sing X$ different from $P$.
Let $P'=(x':y':z':t':w')$. If $w'= 0$, then we may assume that 
$P'=(1:0:0:0:0)$, so that $y$, $z$, $t$ and $w$ are local 
coordinates in an affine neighborhood of $P'$. Note that  
all second partial derivatives of the left hand side
of~\ref{equation:Eckardt-point} with respect to $w$ and some other
coordinate out of $y$, $z$, $t$, $w$ vanish at $P'$ (since $q_2$ does), 
so $P'$ cannot be an ordinary double point of $X$. Hence $w'\neq 0$,
and the point 
$$P''=\tau_P(P')=(x':y':z':t':-w')$$ 
is different from $P$ and $P'$; it lies on $L$ and is singular on $X$.
\end{proof}

Now we will analyze the cases when the involutions $\tau_P$ and $\tau_L$
are regular.

\begin{lemma}\label{lemma:line-regularisation}
Let $L\subset X$ be a line passing through one or two singular
points of $X$. 
Assume that $L$ is not an Eckardt line.
Then the involution $\tau_L$ is not regular.
\end{lemma}
\begin{proof}
This follows from Lemma~\ref{lemma:line-1-point-action} in the case of one 
singular point on $L$ and from Lemma~\ref{lemma:line-2-points-action}
in the case of two singular points.
\end{proof}

\begin{lemma}\label{lemma:point-regularisation}
Let $P\in\Sing X$. Then the involution $\tau_P$ is regular if and only if
$P$ is an Eckardt point on $X$.
\end{lemma}
\begin{proof}
If $P$ is an Eckardt point, $\tau_P$ is regular by
Lemma~\ref{lemma:Eckardt-point-equation} and
Example~\ref{example:Eckardt-point}. Now assume that $P$ is not an Eckardt
point. Then a general line $L\subset\P^4$ such that
$\mult_P\left(\left. X\right|_L\right)\ge 3$ is not contained in $X$,
and $\mult_P\left(\left. X\right|_L\right)=3$. So there is a single
intersection point $P_L\in X\cap L$ that is different from $P$, 
and hence $\tau_P$
is not regular at $P$ (equivalently, one can see that the divisor $D$
swept out by such points $P_L$ maps to $P$ under~$\tau_P$).
\end{proof}

\begin{remark}
If $P$ is a point such that there is a non-Eckardt line $L\subset X$ passing 
through $P$, then one can use Lemmas~\ref{lemma:point-action-1-point},
\ref{lemma:point-action-2-points} and~\ref{lemma:point-action-3-points} to
show that $\tau_P$ is non-regular. Still, the direct proof 
of Lemma~\ref{lemma:point-regularisation} seems more convenient 
since it avoids us having to look for such line passing through~$P$.
\end{remark}

Combining the previous results we get the following.

\begin{corollary}\label{corollary:line-regularisation}
An involution $\tau_L$ is regular if and only if $L$ is an Eckardt line.
\end{corollary}
\begin{proof}
If $L$ is an Eckardt line then either $L$ contains a singular Eckardt point
or there are infinitely many lines contained in $X$ that intersect $L$ in
smooth points of $X$. In the former case $\tau_L$ is regular by 
Remark~\ref{remark:Eckardt-point-easy} or by
Lemma~\ref{lemma:Eckardt-point-equation} and 
Example~\ref{example:Eckardt-point}. In the latter case $\tau_L$ is regular
by Example~\ref{example:Eckardt-line-regularisation}.

If $L$ is not an Eckardt line then $\tau_L$ is not regular by
Lemma~\ref{lemma:line-regularisation}.
\end{proof} 

\begin{proof}[Proof of Proposition~\ref{proposition:regularisation}.]
See Corollary~\ref{corollary:line-regularisation} and 
Lemma~\ref{lemma:point-regularisation}.
\end{proof}

\begin{remark}\label{remark:why-Mella-not-enough}
In~\cite{Mella} it was proved that a non-canonical center on $X$ is either 
a singular point or a line containing one or two singular points.
As we have seen in this section, the involutions $\tau_P$ and 
$\tau_L$ are untwisting involutions for a point $P$ and a line $L$,
respectively, only if $P$ is not an Eckardt point and $L$ is not an Eckardt 
line. This means that to derive Theorem~\ref{theorem:Mella-rigid} from 
the results of~\cite{Mella} one should check that Eckardt points and lines 
cannot be non-canonical centers. This is done below.

An Eckardt point cannot be a maximal center by 
Remark~\ref{remark:Eckardt-point-easy}. Let $L$ be an Eckardt line.
Then either $L$ contains a singular Eckardt point~$P$,
or there are infinitely many lines contained in $X$ that intersect $L$ in
smooth points of $X$. Assume that $L$ is a non-canonical center
with respect to a normalized mobile linear system $\frac{1}{\mu}\H$. 
In the former case take a general plane section containing $L$ and some line 
passing through~$P$. Then a residual conic $Q$ (that is possibly reducible
but does not contain $L$ as a component) intersects $L$ in 
two smooth points of~$X$ (since $L$ cannot contain exactly two singular 
points by Lemma~\ref{lemma:Eckardt-point-equation}) and hence is contained in 
$\Bs\H$\t contradiction. In the latter case a general line 
intersecting $L$ is contained in $\Bs\H$, which is also a contradiction.
\end{remark}

%%%%%%%%%%%%%%%%%%%%%%%%%%%%%%%%%%%%%%%%%%%%%%%%%%%%%%%%%%%%%%%%%%%

\section{Non-canonical centers}
\label{section:max-centers}

From now on we denote by 
$\H$ the linear system obtained as in section~\ref{section:preliminaries}.
Recall that by a non-canonical center we mean a non-canonical center of
$\frac{1}{\mu}\H$.

Some of the results of~\cite{Mella} can be summarized as follows.

\begin{theorem}[{see~\cite[Theorem~17]{Mella}}]
\label{theorem:Mella}
A non-canonical center
on $X$ is either a singular point or a line
passing through one or two singular points.
\end{theorem}

One of the purposes of this section is to prove the following.

\begin{proposition}\label{proposition:2-centers}
Assume that there are at least two non-canonical centers appearing
simultaneously on $X$.
Then there are exactly two of them and they are either
two singular points connected by a line contained in~$X$, or a singular
point and a line containing exactly one more singular point.
\end{proposition}

\begin{remark}
By Theorem~\ref{theorem:sing-mult}
an ordinary double point $P$ is a non-canonical center with respect to
$\frac{1}{\mu}\H$ if and only if $\mult_P\H>\mu$.
The same holds for a line $L\subset X$ (or, more generally, for any curve not 
contained in the singular locus of an ambient variety), since the only extremal 
contraction with center in $L$ is isomorphic to the blow-up of $X$ along $L$
in a neighborhood of a general point of $L$. 
\end{remark}

\begin{lemma}
\label{lemma:2-points}
If the points $P_1$ and $P_2$ are non-canonical centers then the line
$L=\ang{P_1, P_2}$ is contained in $X$.
\end{lemma}
\begin{proof}
Assume that $L\not\subset X$. Let $H'$ be a general member of the
linear system $|H-P_1-P_2|$. Then $H'$ does not contain any base curves of
$\H$ and for general $D_1, D_2\in\H$ the local intersection index
$(D_1D_2H')_{P_i}>\nlb 2\mu^2$
by Theorem~\ref{theorem:sing-mult}. Hence
$$4\mu^2=D_1D_2H'\ge (D_1D_1H')_{P_1}+(D_1D_2H')_{P_2} >
2\mu^2+2\mu^2=4\mu^2,$$
a contradiction.
\end{proof}

\begin{lemma}
\label{lemma:3-points}
If the points $P_1$, $P_2$ and $P_3$ are non-canonical 
centers then they are not collinear.
\end{lemma}
\begin{proof}
Assume they are collinear. By Lemma~\ref{lemma:2-points} the line
$L=\ang{P_1, P_2, P_3}$ is contained in $X$. Let $\Pi$ be a general
two-dimensional plane passing through $L$, and $\left.X\right|_{\Pi}=L\cup
C$. Since $C\not\subset\Bs\H$, by Theorem~\ref{theorem:sing-mult}
for a general $D\in\H$ we have 
$$3\mu=CD\ge \sum\nolimits_{i=1}^3\mult_{P_i}\H>\sum\nolimits_{i=1}^3\mu=
3\mu,$$
a contradiction.
\end{proof}

\begin{lemma}
\label{lemma:2-points-and-line}
If the points $P_1$ and $P_2$ are non-canonical centers then the line
$L=\ang{P_1, P_2}$ is not a non-canonical center.
\end{lemma}
\begin{proof}
Similar to that of Lemma~\ref{lemma:3-points}.
\end{proof}

\begin{lemma}
\label{lemma:1-point-and-line}
If a point $P$ and a line $L\ni P$ are non-canonical centers then
$L$ contains exactly one more singular point.
\end{lemma}
\begin{proof}
Similar to that of Lemma~\ref{lemma:3-points} (except for the ``exactly'',
which is implied by Theorem~\ref{theorem:Mella}).
\end{proof}

\begin{lemma}
\label{lemma:2-spatial-lines}
Two skew lines cannot both be non-canonical centers.
\end{lemma}
\begin{proof}
Assume that there exist skew lines $L_1$ and $L_2$
that are non-canonical centers. Let $\Pi$ be a general plane passing through
$L_1$, and  $\left.X\right|_{\Pi}=L_1\cup C$. Let
$C\cap L_1=\{P_1, P_2, P_3\}$, $C\cap L_2=P$.
By Theorem~\ref{theorem:Mella} at least one of the points
$P_1$, $P_2$, $P_3$ is a nonsingular point of $X$. Since $P$ is also
nonsingular and $C\not\subset\Bs\H$, for a general $D\in\H$
we have
$$3\mu=CD\ge
\mult_P\H+\sum\nolimits_{i=1}^3\mult_{P_i}\H>
\mu+\mu+\frac{\mu}{2}+\frac{\mu}{2}=3\mu,$$
a contradiction.
\end{proof}

\begin{lemma}\label{lemma:2-points-on-Eckardt-line}
Let the points $P_1$ and $P_2$ be non-canonical centers. Assume that 
the line $L=\ang{P_1, P_2}$ does not pass through any 
other singular points of $X$. Then $L$ is not an Eckardt line.
\end{lemma}
\begin{proof}
Assume that it is an Eckardt line 
(note that $L\subset X$ by Lemma~\ref{lemma:2-points}).
Let $L'\subset X$ be a general line intersecting $L$, $\Pi=\ang{L, L'}$ and
let $\left.\Pi\right|_X=L+L'+Q$, where $L\not\subset Q$ by 
Lemma~\ref{lemma:tangent-plane}. Then $Q$ is a (possibly reducible) conic
passing through $P_1$ and $P_2$, so by Theorem~\ref{theorem:sing-mult}
it is contained in $\Bs\H$, a contradiction.
\end{proof}

\begin{lemma}\label{lemma:2-points-collinear-to-Eckardt-point}
Let the points $P_1$ and $P_2$ be non-canonical centers. Assume that
the line $L=\ang{P_1, P_2}$ contains a third singular point $P_3$.
Then $P_3$ is not an Eckardt point.
\end{lemma}
\begin{proof}
Analogous to that of Lemma~\ref{lemma:2-points-on-Eckardt-line}.
Note that in this case a general residual conic $Q$ does not contain 
$L$ because the cone of lines passing through an Eckardt point is not 
contained in a hyperplane by Lemma~\ref{lemma:Eckardt-point-equation}.
\end{proof}

Lemma~\ref{lemma:negative-definite} below is our main tool to exclude
configurations of non-canonical centers. To state it we will use the following 
notations.

Let the lines $C_1, \ldots, C_k\subset X$, $0\le k\le 4$, and the points
$P_1, \ldots, P_l\in\Sing X$, $l\ge 0$, be contained in a plane $\Pi_0$.
Let
$$\left.X\right|_{\Pi_0}=d_1C_1+\ldots+d_kC_k+\ldots+d_mC_m$$
for some  $m\le 4$, and
$$\Pi_0\cap\Sing X=\{P_1, \ldots, P_l, P_{l+1}, \ldots, P_n\}.$$
Let $H$ be a general hyperplane section passing through $\Pi_0$, so that
$$\Sing H=\{P_1, \ldots, P_n, P_{n+1}, \ldots, P_r\},$$ 
where $r\ge n$ (note
that by Lemma~\ref{lemma:general-section-mult-line}
the inequality $r>n$ can hold only if the intersection
$X\cap\Pi_0$ has components with multiplicities greater than~$1$).
Let $\ov{\pi}:\tilde{X}\to X$ be a sequence of blow-ups
with centers lying over the points $P_1, \ldots, P_r$ such that the
restriction $\pi$ of $\ov{\pi}$ to the strict transform
$\tilde{H}$ of $H$ is a minimal resolution of $H$.
Let $\ov{E_i^t}$ be exceptional divisors of $\ov{\pi}$ such that
$\ov{\pi}(\ov{E_i^t})=P_i$ for $1\le i\le r$, $1\le t\le\ov{T_i}$;
let 
$E_i^t$, $1\le i\le r$, $1\le t\le T_i$, be the components of 
the restrictions to $\tilde{H}$ of the divisors $\ov{E_i^t}$ (so the 
$E_i^t$ are prime exceptional divisors of $\pi$ with $\pi(E_i^t)=P_i$;
note that $T_i$ may be different from $\ov{T_i}$); 
finally, let $\tilde{C_j}$ be the proper transforms of 
$C_j$ for $1\le j\le m$.

\begin{lemma}\label{lemma:negative-definite}
Let $(\cdot, \cdot)$ be the intersection form on $\NS^1(\tilde{H})$.
Let $G$ be the set of all curves $E_i^t$, $l+1\le i\le r$, and
$\tilde{C_j}$, $k+1\le j\le m$, and $G'$ the set of all curves
$E_i^t$, $1\le i\le l$, and $\tilde{C_j}$, $1\le j\le k$. 
Assume that the following condition holds:

$(*)$ the set $G$ splits into a disjoint union $G=G_1\cup\ldots\cup G_p$
such that for all $1\le s\le p$ the intersection form $(\cdot, \cdot)$ 
is negative semi-definite on the subspace $W_s$ generated by 
$G_s$, negative definite on each subspace of $W_s$ generated by all elements 
of $G_s$ except one, and the subspaces $W_s$ 
are pairwise orthogonal with respect to $(\cdot, \cdot)$. 

Then all curves from $G'$ cannot appear simultaneously as non-canonical centers 
on $X$.
\end{lemma}

\begin{remark}\label{remark:diagrams}
Lemma~\ref{lemma:negative-definite} will be applied to normal crossing
configurations of nonsingular rational curves on K3 surfaces. Such a curve
is a $(-2)$-curve, so the properties of the corresponding intersection
form depend only on the structure of a dual graph (and the condition
of Lemma~\ref{lemma:negative-definite} is equivalent to the requirement
that
all connected components of the dual graph are subgraphs of affine
Dynkin diagrams). To describe such graphs
we will use the standard notation for usual and affine Dynkin diagrams
(see, for example,~\cite{Kac}).
\end{remark}

We start with two simple examples to clarify the idea of the proof.
The general situation differs only in minor technical details: one should 
assume that there is a decomposition $(*)$ to allow configurations with
non-connected dual graphs etc.

\begin{example}\label{example:3-pt-0-pt}
Let $P_1$, $P_2$ and $P_3$ be non-collinear singular points of~$X$. Let 
$\Pi_0=\ang{P_1, P_2, P_3}$ and $L_i=\ang{P_j, P_k}$ for $\{i, j, k\}=\{1,
2, 3\}$; let $L_4$ be the residual line 
$$L_4=(X\cap\Pi_0)\setminus (L_1\cup L_2\cup L_3).$$
Assume that $L$ does not pass through any of the points $P_i$. Let
$Q_i=\nlb L_4\cap\nlb L_i$. 
Assume also that the points $Q_i$ are nonsingular on~$X$. 
Then in the notations of Lemma~\ref{lemma:negative-definite} the
surface $H$ has nodes at the points~$P_i$ and is nonsingular outside $P_i$
(one can apply Lemmas~\ref{lemma:general-section}
and~\ref{lemma:general-section-mult-line}, but in this particular case it
is actually much easier to see).
Let $E_i$ be exceptional divisors over $P_i$ on the minimal resolution 
$\pi:\tilde{H}\to H$. 

Let us prove that the points $P_i$ cannot appear simultaneously as
non-canonical centers on $X$. Assume that they can. Then 
(see a calculation in the general case in the proof 
of Lemma~\ref{lemma:negative-definite} below) one has
\begin{equation}\label{eq:3-pt-0-pt}
F+\sum\limits_{i=1}^{3}\kappa_i E_i\numeq \sum\limits_{j=1}^{4}
\theta_j'\tilde{L_j}
\end{equation}
for some mobile divisor $F$, some strictly positive coefficients $\kappa_i$ 
and non-negative coefficients $\theta_j'$. It is easy to see that 
$$(F+\sum\limits_{i=1}^{3}\kappa_i E_i)(\sum\limits_{j=1}^{4}
\theta_j'\tilde{L_j})\ge 0,$$
since both parts of~\ref{eq:3-pt-0-pt} are effective and do not have
common components. On the other hand, $\tilde{L_j}$ are $(-2)$-curves on a
$K3$ surface $\tilde{H}$, and the dual graph of the corresponding
configuration is of type $D_4$. Hence the intersection form 
on the subspace $W\subset\NS^1(\tilde{H})$ generated by the curves
$\tilde{L_j}$ is negative definite. 
The latter implies that the self-intersection of the right
hand side of~\ref{eq:3-pt-0-pt} can be non-negative only if all $\theta_j'$
vanish. But this is impossible since an effective divisor cannot be
numerically trivial, a contradiction.
\end{example}

\begin{example}\label{example:3-pt-3-pt}
In the setting of Example~\ref{example:3-pt-0-pt} assume that all the points 
$Q_i$ are singular on $X$. Then $H$ has nodes at $P_i$ and $Q_i$ and is
nonsingular outside these points. Let $F_i\subset\tilde{H}$ 
be exceptional divisors over the points $Q_i$. 

Let us prove that in this case 
the points $P_i$ cannot appear simultaneously as
non-canonical centers on $X$. Assume that they can. Then
\begin{equation}\label{eq:3-pt-3-pt}
F+\sum\limits_{i=1}^{3}\kappa_i E_i\numeq 
\sum\limits_{i=1}^{3}\kappa_i' F_i+
\sum\limits_{j=1}^{4} \theta_j'\tilde{L_j}
\end{equation}
for some mobile divisor $F$, some strictly positive coefficients $\kappa_i$
and non-negative coefficients $\kappa_i'$ and $\theta_j'$. Again we have
\begin{equation}\label{eq:3-pt-3-pt-ge0}
(F+\sum\limits_{i=1}^{3}\kappa_i E_i) 
(\sum\limits_{i=1}^{3}\kappa_i' F_i+
\sum\limits_{j=1}^{4} \theta_j'\tilde{L_j})\ge 0.
\end{equation}
Note that $\tilde{L_j}$ and $F_i$ are $(-2)$-curves on a $K3$ surface
$\tilde{H}$, and the dual graph of the corresponding configuration is of
type $E_6^{(1)}$. 
In particular, the self-intersection of the right hand side
of~\ref{eq:3-pt-3-pt} is non-positive, and hence it vanishes
by~\ref{eq:3-pt-3-pt-ge0}.
Since the right hand side of~\ref{eq:3-pt-3-pt} 
cannot be zero, for its self-intersection to be zero it is necessary that 
\emph{all} the coefficients $\kappa_i'$ and $\theta_j'$ 
should be strictly positive. But in the latter case the intersection of the
left and the right hand sides of~\ref{eq:3-pt-3-pt} is strictly positive,
since $\kappa_1>0$, $\theta_2'>0$ and $E_1\tilde{L_2}>0$, a contradiction.
\end{example}

\begin{proof}[{Proof of Lemma~\ref{lemma:negative-definite}.}]
Assume that they can. Let $\mult_{C_j}\H=\gamma_j$.
Let $H'$ be a general hyperplane section passing through $\Pi_0$; then
$\left.H'\right|_H=C_1+\ldots+ C_m$.
Since the singularities of $H$ are Du Val of type $A$
(see Lemmas~\ref{lemma:general-section}
and~\ref{lemma:general-section-mult-line}),
we have
$$\pi^*(\left.H'\right|_H)=\pi^{-1}(\left.H'\right|_H)+
\sum\limits_{i=1}^r\sum\limits_{t=1}^{T_i}E_i^t.$$

Let $\ov{\H}=\ov{\pi}^{-1}\H$. Define $\nu_i^t$ to satisfy
$$\ov{\H}=\ov{\pi}^*\H-
\sum\limits_{i=1}^r\sum\limits_{t=1}^{\ov{T_i}}\nu_i^t\ov{E_i^t}.$$
Note that since $H$ has only Du Val singularities of type $A$, all divisors 
$\left.\ov{E_i^t}\right|_{\tilde{H}}$ are reduced, and hence
$$\left.\left(\sum\limits_{t=1}^{\ov{T_i}}\ov{E_i^t}\right)\right|_{\tilde{H}}=
\sum\limits_{t=1}^{T_i} E_i^t.$$

Let
$$\left.\ov{\H}\right|_{\tilde{H}}=
F+\sum\limits_{j=1}^m\gamma_j \tilde{C_j},$$
where $F$ is a mobile divisor. Then
\begin{multline}
\label{equality:1}
F+\sum\limits_{j=1}^m\gamma_j\tilde{C_j}=
\left.\ov{\H}\right|_{\tilde{H}}=
\left.\left(\ov{\pi}^*\H-
\sum\limits_{i=1}^r\sum\limits_{t=1}^{\ov{T_i}}\nu_i^t\ov{E_i^t}
\right)\right|_{\tilde{H}}\numeq\\
\numeq\left.(\ov{\pi}^*(\mu H'))\right|_{\tilde{H}}-
\sum\limits_{i=1}^r
\sum\limits_{t=1}^{\ov{T_i}}\nu_i^t\left.\ov{E_i^t}\right|_{\tilde{H}}=\\
=\pi^*(\mu\left.H'\right|_{H})-
\sum\limits_{i=1}^r\sum\limits_{t=1}^{T_i}\nu_i^tE_i^t=\\
=\mu\pi^{-1}(\left.H'\right|_H)+\mu\sum\limits_{i=1}^r
\sum\limits_{t=1}^{T_i}E_i^t-
\sum\limits_{i=1}^r\sum\limits_{t=1}^{T_i}\nu_i^tE_i^t=\\
=\mu\sum\limits_{j=1}^m\tilde{C_j}+\sum\limits_{i=1}^r\sum\limits_{t=1}^{T_i}
(\mu-\nu_i^t)E_i^t.
\end{multline}

Rewrite the equality~\ref{equality:1} as
\begin{equation}
\label{equality:2}
F+\sum\limits_{i, t}\kappa_i^tE_i^t+\sum\limits_j\theta_j\tilde{C}_j\numeq
\sum\limits_{i', t'}\kappa_{i'}^{t'}E_{i'}^{t'}+
\sum\limits_{j'}\theta_{j'}\tilde{C}_{j'},
\end{equation}
where all the coefficients $\kappa_i^t$, $\kappa_{i'}^{t'}$, $\theta_j$
and $\theta_{j'}$ are positive, and the sets of summation indices of
the right hand side and the left hand side are disjoint.
By assumption $\mult_{P_i}\H>\mu$ for $1\le i\le l$; in particular,
$\nu_i^t>\nolinebreak\mu$
for $1\le i\le l$. By assumption we also have $\gamma_j>\mu$
for $1\le j\le k$. (We do not assume a priori that  
$\nu_i^t\le\mu$ for $l+1\le
i\le r$ or that $\gamma_j\le\mu$ for $k+1\le j\le m$.)
We do not exclude the possibility that some summations in~\ref{equality:2}
are performed over empty sets of indices, but in any case 
the set of indices $i'$
(resp., $j'$) that appear on the right hand side of~\ref{equality:2} 
is contained in the set
$\{l+1, \ldots, r\}$ (resp., $\{k+1, \ldots, m\}$) by the assumption on 
multiplicities.
Condition~$(*)$ implies that the intersection form
is negative semi-definite on the space $W=\bigoplus_s W_s$, 
so by Lemma~\ref{lemma:semi-definite}
\begin{equation}
\label{equality:3}
(F+\sum\kappa_i^tE_i^t+\sum\theta_j\tilde{C_j})
(\sum\kappa_{i'}^{t'}E_{i'}^{t'}+\sum\theta_{j'}\tilde{C}_{j'})
=0.
\end{equation}

The right hand side of the equality~\ref{equality:2} is
non-zero since an effective divisor cannot be numerically trivial. 
By~\ref{equality:3} the self-intersection of the right hand side 
of~\ref{equality:2} is zero, so condition $(*)$ implies that 
for any $1\le s\le p$ either all curves from $G_s$ appear on the 
right hand side of~\ref{equality:2} with non-zero coefficients,
or no curve from $G_s$ appears there at all.
The union $\bigcup_{i, t} E_i^t\cup\bigcup_{j}\tilde{C_j}$ is connected,
and by condition $(*)$ any two curves $D_1\in G_{s_1}$, $D_2\in G_{s_2}$ 
are disjoint for $s_1\neq s_2$. Hence for any $1\le s\le p$ there are curves
$D\in G_s$ and $D'\in G'$ such that $D$ intersects $D'$. 
Since all the curves $D'\in G'$
appear on the left hand side of~\ref{equality:2} with non-zero coefficients, 
the intersection of the left hand side and the right hand side 
of~\ref{equality:2} is strictly positive; this contradicts~\ref{equality:3}.
\end{proof}

\begin{corollary}\label{corollary:3-points}
Three points cannot appear simultaneously as non-canonical centers on $X$.
\end{corollary}
\begin{proof}
Assume that the points $P_1$, $P_2$ and $P_3$ are non-canonical centers. By
Lemma~\ref{lemma:3-points} they are not collinear, and by
Lemma~\ref{lemma:2-points} the lines $L_{ij}=\ang{P_i, P_j}$ are contained
in $X$. Let $\Pi_0=\ang{P_1, P_2, P_3}$. Then
$$\left.X\right|_{\Pi_0}=L_{12}+L_{23}+L_{13}+L,$$ 
where $L$ is a line
(possibly coinciding with one of the lines $L_{ij}$). Let
$\pi:\tilde{H}\to H$ be a minimal resolution
of singularities of a general hyperplane section $H$ passing through
$\Pi_0$. Let $G$ be the collection of proper transforms of $L$ and
$L_{ij}$, and of all exceptional curves of
$\pi$ except those that lie over the points $P_i$. Let $\Gamma$ be the dual
graph of~$G$.

If $L$ coincides with one of the lines $L_{ij}$ (say, with $L_{12}$), then
by Lemmas~\ref{lemma:general-section}
and~\ref{lemma:general-section-mult-line}
the surface $H$ has at worst $A_2$ singularities at $P_1$ and~$P_2$ and~$A_1$ 
singularities at $P_3$ and possibly at one
more point $P\in L_{12}$. One easily checks that the only component 
of $\Gamma$ that is not a point is of type $A_2$.

If $L$ coincides with none of the lines $L_{ij}$ but passes through one of
their intersection points $P_i$, say through $P_1$, then
by Lemma~\ref{lemma:general-section} the surface $H$ has at worst an
$A_2$ singularity at $P_1$, singularities of type $A_1$ at the points
$P_2$ and $P_3$ and possibly one more $A_1$ singularity at the point
$P=L\cap L_{23}$
(if $X$ itself is singular at $P$). So $\Gamma$ is the union of two
graphs that consists of single points 
with a graph of type $A_3$ or $A_2$, depending on
whether $X$ is singular at $P$ or not.

If $L$ passes through none of the points $P_i$ then
by Lemma~\ref{lemma:general-section} all singularities of $H$ are of type
$A_1$ and $\Gamma$ is a subgraph of a graph of type~$E_6^{(1)}$ (cf
Examples~\ref{example:3-pt-3-pt} and~\ref{example:3-pt-0-pt}).

In any case the intersection form on the subspace
$W\subset\NS^1(\tilde{H})$
generated by $G$ satisfies the conditions of
Lemma~\ref{lemma:negative-definite}; hence $P_1$, $P_2$ and $P_3$
do not appear simultaneously as non-canonical centers.
\end{proof}

\begin{corollary}\label{corollary:2-lines}
Two lines cannot appear simultaneously as non-canonical centers on $X$.
\end{corollary}
\begin{proof}
Assume that the lines $L_1$ and $L_2$ are non-canonical centers. By
lemma~\ref{lemma:2-spatial-lines} they are coplanar.
Let $\Pi_0=\ang{L_1, L_2}$. Then
$$\left.X\right|_{\Pi_0}=L_1+L_2+Q,$$ 
where $Q$ is a (possibly reducible) conic.
Let $\pi:\tilde{H}\to H$ be a minimal resolution of singularities of a
general hyperplane section $H$ passing through $\Pi_0$.
Let $G$ be the collection of proper transforms of the components of $Q$
and all exceptional curves of $\pi$. Let $\Gamma$ be the dual graph.

If the conic $Q$ is irreducible then the  
only component of $\Gamma$ that is not a point
(such a component exists if $Q$ contains singularities of $X$)
is a subgraph of a graph of type $D_5$ or $D_4^{(1)}$, depending on
whether $Q$ passes through the point $P=L_1\cap L_2$ or not
(in the former case by Lemma~\ref{lemma:general-section}
there are at most two singularities of type $A_1$ and one of type $A_2$ on
$Q\subset H$, and in the latter case there are at most four singularities
of type $A_1$).

If $Q=L_3+ L_4$, $L_3\neq L_4$, $L_3\not\ni P$, $L_4\not\ni P$ and the point
$P'=L_3\cap L_4$ lies neither on $L_1$ nor on $L_2$, 
then by Lemma~\ref{lemma:general-section} the surface $H$ has only $A_1$
singularities, and the only component of $\Gamma$ that is not a point
is a subgraph of a graph of type $D_6^{(1)}$ or $D_5^{(1)}$
depending on whether the point
$P'=L_3\cap L_4$ is singular on $X$ or not.

If $Q=L_3+ L_4$, $L_3\neq L_4$, $L_3\not\ni P$, $L_4\not\ni P$ and the point
$P'=L_3\cap L_4$ lies on $L_1$, then by Lemma~\ref{lemma:general-section} 
the surface $H$ has only $A_1$ singularities except for a possible 
$A_2$ singularity at $P'$, and the only component 
of $\Gamma$ that is not a point is a subgraph of a graph of type $E_6$.

If $Q=L_3+ L_4$, $L_3\neq L_1$, $L_3\neq L_2$, $L_3\ni P$, $L_4\not\ni P$,
then by Lemma~\ref{lemma:general-section} the surface $H$ has only $A_1$
singularities except for a possible singularity of type $A_2$ at the point
$P$, and the only component of $\Gamma$ that is not a point
is a subgraph of a graph of type $D_7$.

If $Q=L_3+ L_4$, the lines $L_i$ are distinct for $1\le i\le 4$,
and $L_3, L_4\ni P$,
then by Lemma~\ref{lemma:general-section} the surface $H$
is only singular at the point $P$ and this singularity is at worst
$A_3$, so the only component of $\Gamma$ that is not a point
is a subgraph of a graph of type $D_5$.

If $Q=2L$, $L\not\ni P$, then by Lemmas~\ref{lemma:general-section}
and~\ref{lemma:general-section-mult-line} the surface
$H$ has at worst $A_2$ singularities at the points $L\cap L_i$ and possibly
one more singularity of type $A_1$ at some point $P'\in L$;
the only component of $\Gamma$ that is not a point
is a subgraph of a graph of type $E_6$.

If $Q=2L$, $L\neq L_i$, $L\ni P$, then by
Lemmas~\ref{lemma:general-section}
and~\ref{lemma:general-section-mult-line} the surface
$H$ has at worst an $A_3$ singularity at the point $P$ and at most two
singularities of type $A_1$ at some points $P', P''\in L$;
the only component of $\Gamma$ that is not a point
is a subgraph of a graph of type $D_6$.

If $Q=L_1+L$, $L\not\ni P$, then by Lemmas~\ref{lemma:general-section}
and~\ref{lemma:general-section-mult-line} the surface
$H$ has at worst $A_2$ singularities at the points $P$ and $P'=L\cap L_1$
and possibly one more singularity of type $A_1$ at some point $P''\in
L_1$; the graph $\Gamma$ has at most two components that are not points, one
of type $A_2$ and the other of type $A_k$ with $k\le 4$.

Finally, if $Q=L_1+L$ with $L\ni P$ ($L$ may coincide with $L_1$ or~$L_2$),
then by Lemmas~\ref{lemma:general-section}
and~\ref{lemma:general-section-mult-line} the surface
$H$ has at worst an $A_3$ singularity at the point $P$
(and at worst $A_2$ singularities on multiple lines, the $A_2$ case 
arising only if $L=L_1$), and all components of $\Gamma$ that are not points
are of type $A_k$ with $k\le 4$.

In any case the intersection form on the subspace
$W\subset\NS^1(\tilde{H})$
generated by $G$ satisfies the conditions of
Lemma~\ref{lemma:negative-definite}; hence $L_1$ and $L_2$
do not appear simultaneously as non-canonical centers.
\end{proof}

\begin{corollary}\label{corollary:line-and-point}
A line and a point outside it cannot appear simultaneously as non-canonical
centers on $X$.
\end{corollary}
\begin{proof}
Assume that a line $L$ and a point $P\not\in L$ are non-canonical centers. Let
$\Pi_0=\ang{L, P}$,
$\left.X\right|_{\Pi_0}=L+C$. 
Let $\pi:\tilde{H}\to H$ be a minimal resolution of singularities of a
general hyperplane section $H$ passing through~$\Pi_0$.
Let $G$ be the collection of proper transforms of components of $C$ and
all exceptional curves of $\pi$ except those that lie over $P$. Let
$\Gamma$ be the dual graph.

If $C$ is an irreducible cubic\footnote{
In this case one can also argue as follows, avoiding the use of
Lemma~\ref{lemma:negative-definite}: if $L$ and $P$ are non-canonical centers,
after an involution $\tau_P$ the curve $C$ becomes a non-canonical center
that is impossible by Theorem~\ref{theorem:Mella}.}
(singular at $P$), then $H$ has singularities of
type $A_1$, and $\Gamma$ is a subgraph of a graph of type $D_4$.

If $C=Q+ L_1$, where $Q$ is an irreducible conic, then $L_1\ni P$ 
(in particular, $L_1\neq L$), and
the only component of $\Gamma$ that is not a point
(if any) is a subgraph of a graph of type $D_6$.

If $C=L_1+ L_2+ L_3$ and the lines $L$, $L_1$, $L_2$ and $L_3$ are
distinct and the latter three lines pass through the point $P$,
then by Lemma~\ref{lemma:general-section}
the surface $H$ has only singularities of type $A_1$ outside~$P$,
and $\Gamma$ has at most three components that are not points, each of them
of type $A_2$.

If $C=L_1+L_2+ L_3$, the lines $L$, $L_1$, $L_2$ and $L_3$ are 
distinct, $L_1$ and $L_2$ pass through $P$, and $L_3$ passes through the
intersection point $P_1=\nlb L\cap\nlb L_1$,
then by Lemma~\ref{lemma:general-section} the surface $H$ has only $A_1$
singularities except for a possible $A_2$ singularity at the point $P_1$,
and the only component of $\Gamma$ that is not a point
is a subgraph of a graph of type $D_7$.

If $C=L_1+ L_2+ L_3$, the lines $L$, $L_1$, $L_2$ and $L_3$ are
distinct,
$L_1$~and $L_2$ pass through $P$, and $L_3$ passes neither through $P$,
nor through the intersection points of the lines $L$ and $L_1$ or $L$ and
$L_2$, then by Lemma~\ref{lemma:general-section} the surface
$H$ has only $A_1$ singularities, and the only component 
of $\Gamma$ that is not a point is a subgraph of a graph of type~$E_7^{(1)}$.

If $C=2L_1+L_2$, $P\not\in L_2$ and $L_2\neq L$, then the surface $H$ has 
only $A_1$ singularities except for possible $A_2$ 
singularities at $P$ and $P_1=L\cap L_1$, 
and the only component of $\Gamma$
is a subgraph of a graph of type $E_7$.

If $C=2L+L_1$, $P\in L_2$, $L_2\neq L$, then $\Gamma$ has at most two 
components that are not points, each of type $A_k$ with $k\le 4$.

If $C=2L_1+L$, then the only component of $\Gamma$ that is not a point
is of type $A_k$ with $k\le 5$. 

If $C=3L_1$, then the only component of $\Gamma$ that is not a point
is of type $A_k$ with $k\le 6$.

In any case the intersection form on the subspace
$W\subset\NS^1(\tilde{H})$,
generated by $G$, satisfies the conditions of
Lemma~\ref{lemma:negative-definite}, hence $L$ and $P$
do not appear simultaneously as non-canonical centers.
\end{proof}

\begin{proof}[Proof of Proposition~\ref{proposition:2-centers}]
By Theorem~\ref{theorem:Mella} all non-canonical centers are either lines or
singular points. If one of the centers is a line $L$, then by
Corollary~\ref{corollary:2-lines} all other non-canonical centers are points,
and by Corollary~\ref{corollary:line-and-point} these points lie on $L$;
finally, by Lemma~\ref{lemma:2-points-and-line} there can be at most one
such point, and by Lemma~\ref{lemma:1-point-and-line} the line $L$
contains exactly two singular points. If all non-canonical
centers are points, then by Corollary~\ref{corollary:3-points} there are
only two of them, and by Lemma~\ref{lemma:2-points} they lie on a line
contained in $X$. \end{proof}

\begin{remark}\label{remark:2-centers}
The statement of Proposition~\ref{proposition:2-centers}
(as well as all previous statements)
remains true if instead of two non-canonical centers one considers
a center of non-canonical singularities and a center of strictly canonical
singularities of $\H$. 
\end{remark}

Proposition~\ref{proposition:2-centers} 
(or rather Remark~\ref{remark:2-centers}) implies 
Theorem~\ref{theorem:relations} using the calculations of 
Lemmas~\ref{lemma:line-2-points-action},
\ref{lemma:point-action-2-points} and \ref{lemma:point-action-3-points}
in a standard way (see~\cite[Chapter~V, \S 7]{Manin} 
or~\cite[3.2.4]{IskovskikhPukhlikov} for a very detailed proof).
Note that Lemmas~\ref{lemma:2-points-on-Eckardt-line}
and~\ref{lemma:2-points-collinear-to-Eckardt-point}
ensure that the calculations of the former Lemmas are applicable, i.\,e. 
that for two points $P_1$ and $P_2$ that are non-canonical centers the line
$L=\ang{P_1, P_2}$ is not an Eckardt line if $L$ does not contain a third 
singular point, and that the third singular point is not an Eckardt point if it
does.

\section{Algebraically non-closed fields}
\label{section:non-closed-fields}

One of the results of~\cite{Mella} (namely, \cite[Theorem~5]{Mella})
states that the main theorems
of~\cite{Mella} (birational rigidity of $X$ and description of generators 
of $\Bir(X)$) hold over algebraically non-closed field $\k$ of characteristic 
$0$ as well as over~$\C$. Unfortunately, there is a gap in the proof
(the fact that three conjugate points cannot form a non-canonical center is 
derived from the statement that even two points cannot, 
and this is not true,
see Example~\ref{example:2-points-involution} below).
The aim of this section is to provide a patch for this gap.

\begin{example}[{cf.~\cite[Chapter~V,~1.4]{Manin}}]
\label{example:2-points-involution-closed-field}
Let $P_1, P_2\in\Sing\Xk$ be two points contained in a line $L\subset\Xk$.
Let $E$ be a section of the associated elliptic fibration
arising from the line $L$.
Take a fiberwise reflection in the section $E$, and denote the corresponding
birational involution of $\Xk$ by~$\tau_{P_1P_2}$. If $P_1$ and $P_2$ are both
non-canonical centers then $\tau_{P_1P_2}$ untwists both of them 
(see Lemma~\ref{lemma:2-points-on-Eckardt-line} and 
Lemma~\ref{lemma:2-points-action} below). On the other hand,
starting with the linear system $|\O(1)|$ and taking the strict transform
with respect to $\tau_{P_1P_2}:\Xk\dasharrow \Xk$, one obtains a mobile 
linear system $\H$ such that $P_1$ and $P_2$ are non-canonical centers with 
respect to $\frac{1}{\mu}\H$, provided that $\tau_{P_1P_2}$ is not regular. 
If $X$ is sufficiently general so that $L$ is not an Eckardt line, 
Lemma~\ref{lemma:2-points-action} implies that 
the involution $\tau_{P_1P_2}$ is indeed non-regular.
\end{example}

\begin{example}\label{example:2-points-involution}
Assume that the singular points $P_1$ and $P_2$ are conjugate (i.\,e.
$\{P_1, P_2\}$ is a $\k$-point of $X$ of degree $2$), so that the line
$L=\ang{P_1, P_2}$ is defined over $\k$.
Then the involution $\tau_{P_1P_2}$ is also defined over $\k$. In particular, 
$\{P_1, P_2\}$ can be a non-canonical center on $X$ (provided that $X$ 
is sufficiently general).
\end{example}

\begin{remark}\label{remark:Mella-vs-Manin}
In the setting of Example~\ref{example:2-points-involution} the line
$L$ is defined over~$\k$ and so is the involution $\tau_L$. One has
$$
\tau_{P_1P_2}=\tau_{P_1}\circ\tau_L\circ\tau_{P_2}.
$$
\end{remark}

\begin{lemma}\label{lemma:2-points-action}
Let a line $L\subset X$ contain exactly two singular points $P_1$ and 
$P_2$ of $X_{\ov{\k}}$.
Assume that $L$ is not an Eckardt line.
Then
\begin{gather*}
\mu(\chi\circ\tau_{P_1P_2})=13\mu(\chi)-6\nu_{P_1}(\chi)-6\nu_{P_2}(\chi),\\
\nu_{P_1}(\chi\circ\tau_{P_1P_2})=
14\mu(\chi)-7\nu_{P_1}(\chi)-6\nu_{P_2}(\chi),\\
\nu_{P_2}(\chi\circ\tau_{P_1P_2})=
14\mu(\chi)-6\nu_{P_1}(\chi)-7\nu_{P_2}(\chi),\\
\nu_L(\chi\circ\tau_{P_1P_2})=8\mu(\chi)-4\nu_{P_1}(\chi)-4\nu_{P_2}(\chi)
+\nu_L(\chi).
\end{gather*}
\end{lemma}
\begin{proof}
Analogous to that of Lemma~\ref{lemma:line-1-point-action}.
Note that Remark~\ref{remark:codim-2} is also applicable in this case.
\end{proof}

Lemma~\ref{lemma:2-points-action} implies 
that a point $\{P_1, P_2\}$ of degree $2$ is a non-canonical 
center with respect to some normalized mobile linear system provided that 
the corresponding line $L$ is contained in $X$ and is not an Eckardt line.
In this case the involution $\tau_{P_1P_2}$  
is an untwisting involution for this center 
(again by Lemma~\ref{lemma:2-points-action}).
On the other hand, by Lemma~\ref{lemma:2-points} the point $\{P_1, P_2\}$ 
cannot be a maximal center
if $L$ is not contained in $X$, nor, 
by Lemma~\ref{lemma:2-points-on-Eckardt-line},
if $L$ is an Eckardt line.
Finally, Corollary~\ref{corollary:3-points} applied to 
$X_{\ov{\k}}$ implies the following.\footnote{
One can avoid using Corollary~\ref{corollary:3-points} here since 
this case fits in the setting of either Example~\ref{example:3-pt-3-pt}
or Example~\ref{example:3-pt-0-pt}.}

\begin{corollary}
A $\k$-point of degree $d\ge 3$ cannot be a non-canonical center.
\end{corollary}

So the main statements of~\cite{Mella} 
(i.\,e. Theorem~\ref{theorem:Mella-rigid})
really hold over $\k$. Moreover, the involutions $\tau_{P_1P_2}$ described 
in Example~\ref{example:2-points-involution} are needed only in the proof,
while one does not need to add them to the set of generators since they are 
expressible in terms of the involutions centered in lines and points 
by Remark~\ref{remark:Mella-vs-Manin}.

%%%%%%%%%%%%%%%%%%%%%%%%%%%%%%%%%%%%%%%%%%%%%%%%%%%%%%%%%%%%%%%%%%%%%%%

\end{document}